\documentclass[12pt,reqno]{amsart}
\usepackage{fullpage}

\usepackage[english]{babel}
\usepackage{array,booktabs}

\usepackage{tikz}
\usetikzlibrary{automata}
\usetikzlibrary{positioning}
\usetikzlibrary{arrows,shapes}
\usetikzlibrary{decorations.markings}

\usepackage{graphicx}

\usepackage[pagebackref]{hyperref}
\usepackage[backrefs, alphabetic]{amsrefs}

\numberwithin{equation}{section}
\theoremstyle{plain} 
\newtheorem*{thm*}{Theorem}
\newtheorem{thm}{Theorem}
\newtheorem{prop}[thm]{Proposition}
\newtheorem{lemma}[thm]{Lemma}
\newtheorem{cor}[thm]{Corollary}

\theoremstyle{definition} 
\newtheorem*{defn*}{Definition}
\theoremstyle{remark}

\newcommand{\bbC}{\mathbb{C}}

\newcommand{\bbZ}{\mathbb{Z}}
\newcommand{\bbN}{\mathbb{N}}
\newcommand{\bbV}{\mathbb{V}}
\newcommand{\T}{\mathbb{T}}

\newcommand{\mc}{\mathcal}
\newcommand{\mcZ}{{\mathcal{Z}({\bf n})}}
\newcommand{\mcS}{{\mathcal{S}}}
\newcommand{\mcP}{{\mathcal{P}}}
\newcommand{\mcG}{{\mathcal{G}}}
\newcommand{\mcH}{{\mathcal{H}}}
\newcommand{\mcV}{{\mathcal{V}}}
\newcommand{\mcE}{{\mathcal{E}}}
\newcommand{\mcO}{{\mathcal{O}}}

\newcommand{\mcPr}{{\mathcal{P}_\mathbb{R}}}

\newcommand{\la}{\lambda}

\newcommand{\be}{\beta}
\newcommand{\ga}{\gamma}

\newcommand{\ot   }{\otimes}
\renewcommand{\bot}{{\bigotimes}}
\newcommand{\dual}[1]{{\left( #1 \right)}^*}

\newcommand{\GL}{\mbox{GL}}
\newcommand{\G}{{\mbox{G}({\bf n})}}

\newcommand{\V}{{V({\bf n})}}
\newcommand{\mcT}[1]{{{\mathcal T}^{#1}({\bf n})}}
\newcommand{\bfn}{{\bf n}}
\newcommand{\bfb}{{\bf b}}
\newcommand{\bfs}{{\bf s}}

\newcommand{\End}{{\mbox{End}}}
\newcommand{\twoline}[2]{\genfrac{}{}{0pt}{}{#1}{#2}}

\begin{document}

\title{The measurement of quantum entanglement \\
and enumeration of graph coverings}
\date{28 March 2010}

\author{Michael W. Hero}
\author{Jeb F. Willenbring}
\author{Lauren Kelly Williams}

\address{University of Wisconsin - Milwaukee \\ Department of Mathematical Sciences \\ P.O. Box 0413, Milwaukee WI 53201}
\thanks{This research was supported in part by the University of Wisconsin - Milwaukee, Research Growth Initiative Grant, and by National Security Agency grant \# H98230-09-0054.}

\dedicatory{It is our honor to dedicate this article to Gregg Zuckerman.}

\begin{abstract}

We provide formulas for invariants defined on a tensor product of defining representations of unitary groups, under the action of the product group. This situation has a physical interpretation, as it is related to the quantum mechanical state space of a multi-particle system in which each particle has finitely many outcomes upon observation. Moreover, these invariant functions separate the entangled and unentangled states, and are therefore viewed as measurements of quantum entanglement.

When the ranks of the unitary groups are large, we provide a graph theoretic interpretation for the dimension of the invariants of a fixed degree. We also exhibit a bijection between isomorphism classes of finite coverings of connected simple graphs and a basis for the space of invariants. The graph coverings are related to branched coverings of surfaces.

\end{abstract}

\maketitle

\section{Introduction}

Understanding the orbit structure of a group action is among the central themes of mathematics.  That is, if a group $G$ acts on a set $\mc X$, one wishes to parameterize the set $\mc X/G = \{ \mcO_G(x) | x \in \mc X \}$, where $\mcO_G(x) = \{g \cdot x | g \in G\}$, in a natural way.  This paper addresses a specific situation in this broad theme:
Let $\mcH_1, \cdots, \mcH_r$ denote finite dimensional Hilbert spaces with $\dim \mcH_i = n_i$.  An unsolved problem is to explicitly parameterize the orbits in the tensor product, $\mcH = \mcH_1 \ot \cdots \ot \mcH_r$, under the action of the product of unitary groups, $U(\mcH_1) \times \cdots \times U(\mcH_r)$, given by
\[
    (u_1, \cdots, u_r) \cdot (v_1 \ot \cdots \ot v_r) = (u_1 v_1) \ot \cdots \ot (u_r v_r)
\] where $u_i \in U(\mcH_i)$ are unitary operators and $v_i \in \mcH_i$ for $i=1, \cdots, r$.

A slightly simpler, but still open, problem is to describe a set $\mathcal F$ of functions $f : \mcH \rightarrow \bbC$, which are invariant under the group $U(\mcH_1) \times \cdots \times U(\mcH_r)$ and separate the orbits.  That is, two tensors $x, y \in \mcH$ are in the same orbit if and only if $f(x)=f(y)$ for all $f \in \mathcal F$.
Such a set $\mathcal F$ does indeed exist in the algebra of polynomial functions on the underlying real vector space of $\mcH$ (see \cite{MW02}).

The motivation for studying this particular group action goes back to \cite{EPR}, and in the literature is often described in context with the physical effect known as ``quantum entanglement'', which has gained enormous popularity as the effect suggests vastly improved models of computation (see \cite{Feynman}).

In line with this nomenclature, the invariant functions on $\mcH$ are called ``measurements of quantum entanglement''.  The primary purpose of the present article is to point out some additional combinatorial/geometric structure related to an earlier work by the first two authors in \cite{HeroWillenbring}.
Specifically, the enumeration problems addressed in \cite{HeroWillenbring} can be translated into enumeration problems addressed in \cite{graph_enum}.
We recall the situation briefly and then provide some examples illustrating a correspondence between coverings of simple graphs and the measurements of quantum entanglement.

\subsection{General Setup}
We recall a general situation which includes the above problem.
Let $K$ be a compact Lie group acting $\bbC$-linearly on a
finite dimensional complex vector space $V$.  It is a difficult
problem in representation theory to provide a description of the
$K$-orbits in $V$.  One approach set out in \cite{MW02} and
\cite{W05} is to use the invariant theory of $K$ to separate orbits.
More precisely: Set $\mcPr(V)$ to be the algebra of complex valued
polynomial functions on the vector space $V$ when viewed as a real
vector space.  The group $K$ acts in the standard way on $\mcPr(V)$
by $g \cdot f(v) = f(g^{-1}v)$ for $g \in K$, $f \in \mcPr(V)$ and
$v \in V$.  Let the algebra of $K$-invariants in $\mcPr(V)$ be
denoted by $\mcPr(V)^K$. We have:

\begin{thm*} (c.f. Thm. 3.1 of \cite{MW02})
If $v,w \in V$ then $f(v) = f(w)$ for all $f \in \mcPr(V)^K$ if and
only if $\mcO_K(v) = \mcO_K(w)$; that is, $v$ and $w$ are in the same $K$-orbit.
\end{thm*} 

\bigskip
Fix a sequence of positive integers $\bfn = (n_1, \cdots, n_r)$.
Let  \[ V(\bfn) = \bbC^{n_1} \ot \bbC^{n_2} \ot \cdots \ot
\bbC^{n_r} \] be the representation of $K(\bfn) = \prod_{i=1}^r U(n_i)$
under the standard action on each tensor\footnote{Here we tensor
over $\bbC$.} factor. (Here $U(n)$ denotes the group of $n \times n$ unitary\footnote{A \emph{unitary} matrix, $u$, is an invertible complex matrix s.t. $\overline{u}^t = u^{-1}$.} matrices.)

Well known results of Hilbert establish that the $K(\bfn)$-invariant
subalgebra of $\mcPr(V(\bfn))$ is finitely generated. In spite of this
result, our situation lacks a complete description of such
generators, except for certain small values of the parameter space,
$\bfn = (n_1, n_2, \cdots, n_r)$.  We do not solve this problem here, but make an encouraging first step:  We provide formulas for a set of polynomials that span the vector space of $K(\bfn)$-invariants in $\mcPr(\V)$.  Within a certain ``stable range'' this spanning set is linearly independent.

The $K(\bfn)$-invariant subalgebra inherits a gradation from
$\mcPr(V(\bfn))$.  Thus, let $\mcPr^d(V(\bfn))$ denote the subspace of degree
$d$ homogeneous polynomial functions contained in $\mcPr(V(\bfn))$.  We set
$\mcPr^d(V(\bfn))^{K(\bfn)} = \mcPr^d(V(\bfn)) \cap \mcPr(V(\bfn))^{K(\bfn)}$.  One can see easily that for $d$ odd, $\mcPr^d(V(\bfn))^{K(\bfn)} = (0)$ (see Lemma \ref{lemma_equal_01}).
However, the dimension of $\mcPr^d(V(\bfn))^{K(\bfn)}$ for even $d$ is more
subtle.  Set $h_m(\bfn) = \dim \mcPr^{2m}(V(\bfn))^{K(\bfn)}$.  For fixed $\bfn$, the formal power series in $q$,
\[
    h_0(\bfn) + h_1(\bfn) q + h_2(\bfn) q^2 + \cdots
\] is called the \emph{Hilbert series} of the $K(\bfn)$-invariant subalgebra.  As we shall see, calculating these coefficients is a step in the solution to the problem of finding a vector space basis of the invariants.

In \cite{HeroWillenbring}, it is shown that for fixed $d = 2m$ and $r$ the value of $h_m(\bfn)$ stabilizes as the components of $\bfn$ grow large.
Consequently, we can define
\[
    \widetilde h_{m,r} =
    \lim_{n_1 \rightarrow \infty} \lim_{n_2 \rightarrow \infty}
    \cdots \lim_{n_r \rightarrow \infty} h_m(n_1, \cdots, n_r).
\]
Several papers in the recent literature investigate Hilbert series related to measurements of quantum entanglement.  See, for example, \cites{MW02, W05}.  Despite the fact that the value of $h_m(\bfn)$ is not known in general, the
value of $\widetilde h_{m,r}$ has a surprisingly simple description,
which we present next.

We first set up the standard notation for partitions, which we
define as weakly decreasing finite sequences of positive integers.
We will always use lower case Greek letters to denote partitions. We
will write $\la \vdash m$ to indicate that $\la$ is a partition of
size $m$.  Lastly, if $\la$ has $a_1$ ones, $a_2$ twos, $a_3$ threes etc., let
\[
    z_\la = 1^{a_1} 2^{a_2} 3^{a_3} \cdots a_1! a_2! a_3! \cdots
\]

We have
\begin{thm*} (c.f. Thm. 1.1 of \cite{HeroWillenbring}) For any integers $m \geq 0$ and $r \geq 1$,
\begin{equation}\label{the_formula}
	\widetilde h_{m,r} = \sum_{\la \vdash m} z_\la^{r-2}
\end{equation}
\end{thm*}

\subsection{A Combinatorial Interpretation.}  Let $S_m$ denote the symmetric group on the set $\{1, \cdots, m\}$.  The $r$-fold cartesian product, denoted,
\[
    S_m^r = \{ \bfs = (\sigma_1, \sigma_2, \cdots, \sigma_r) |  \sigma_i \in S_m \mbox{ for all } i \}
\] is acted upon by $S_m \times S_m$ under the action, $(\alpha, \beta) \cdot \bfs = \alpha \bfs \beta^{-1}$ where 
\[ \alpha \bfs \beta^{-1} = (\alpha \sigma_1 \beta^{-1}, \cdots, \alpha \sigma_r \beta^{-1} ).\]
 The orbits of this group action are the double cosets, $\Delta \backslash S_m^r / \Delta$ where 
\[ \Delta = \{(\sigma, \cdots, \sigma)| \sigma \in S_m\}. \]
Next, we shall see that the number of these double cosets is $\widetilde h_{m,r}$.

In \cite{HeroWillenbring}, it is shown that $\widetilde h_{m,r}$ is the number of orbits under the $S_m$-action of ``simultaneous conjugation'',
\[
    \ga \bfs \ga^{-1} = \left( \ga \sigma_1 \ga^{-1}, \cdots, \ga \sigma_{r-1} \ga^{-1} \right),
\] on $S_m^{r-1}$.  Denote these orbits by $\mcO = S_m^{r-1}/S_m$.  There exists a map $\theta : \Delta \backslash S_m / \Delta \rightarrow \mcO$ defined for $\bfs = (\sigma_1, \cdots, \sigma_m)$ by
\[
    \theta(\Delta \bfs \Delta) = \{ \ga( \sigma_1 \sigma_r^{-1}, \sigma_2 \sigma_r^{-1}, \cdots, \sigma_{r-1} \sigma_r^{-1} ) \ga^{-1} \; | \; \ga \in S_m \}.
\]  It is easy to see that $\theta$ is independent of the representative $\bfs$, and defines a bijective function from $\Delta \backslash S_m / \Delta$ to $\mcO$.

In Section \ref{sec_tensor}, we show how to define a spanning set for the invariant tensors on $\V \oplus \V^*$ parameterized by the set $S_m^r$.  This is a simple consequence of Schur-Weyl duality (see Theorem \ref{SWD}).  Then, in Section \ref{sec_symtensor} we show how this spanning set projects onto the invariants in the symmetric tensors on $\V \oplus \V^*$.  After projecting, many equalities arise, and we show how a spanning set for the invariants is naturally parameterized by $\Delta \backslash S_m^r / \Delta$ and $\mcO$.

\bigskip

The right hand side of Equation \ref{the_formula} may be interpreted in certain graph enumeration problems, which we recall following \cite{graph_enum}. Let $\mcG=\mcG(\mcV,\mcE)$ be a simple connected graph, with vertex set $\mcV$ and edge set $\mcE$. Let $\be(\mcG) = |\mcE|-|\mcV|+1$, which is the number of independent cycles in $\mcG$ (the first Betti number). Let $N(v)$ denote the neighborhood\footnote{The neighborhood $N(v)$ of $v \in \mcV$ is the set of all vertices in $\mcV$ adjacent to $v$.} of a vertex $v \in \mcV$.
A graph $\tilde{\mcG}$ is said to be a covering of $\mcG$ with projection $p: \tilde{\mcG} \to \mcG$ if there exists a surjection $p: \tilde{\mcV} \to \mcV$ such that $p|_{N(\tilde{v})}: N(\tilde{v}) \to N(v)$ is a bijection for any $v \in \mcV$ and $\tilde{v} \in p^{-1}(v)$. If $p$ is $n$-to-one, we say $p:\tilde{\mcG} \to \mcG$ is an $n$-fold covering. In the image below, $\tilde{\mcG}$ is a 2-fold covering of $\mcG$. We see that the neighborhood of a green vertex of $\tilde{\mcG}$ maps injectively onto the neighborhood of the green vertex of $\mcG$.

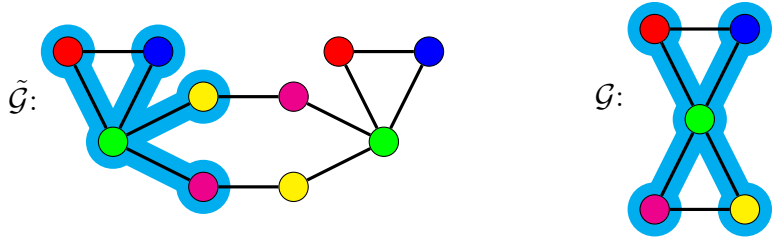
\begin{center}
\begin{tikzpicture}[scale=0.6]
\fill[cyan] (-1,2) circle (6mm); \fill[cyan] (1,2) circle (6mm); \fill[cyan] (0,0) circle (6mm); \fill[cyan] (2,1) circle (6mm); \fill[cyan] (2,-1) circle (6mm);
\node(a)[fill=red,circle,draw] at (-1,2) {}; \node(b)[fill=blue,circle,draw] at (1,2) {}; \node (c)[fill=green,circle,draw] at (0,0) {};
\node(d)[fill=yellow,circle,draw] at (2,1) {}; \node(e)[fill= magenta,circle,draw] at (4,1) {}; \node(f)[fill=green,circle,draw] at (6,0) {}; \node(g)[fill=red,circle,draw] at (5,2){}; \node(h)[fill=blue,circle,draw] at (7,2) {}; \node(i)[fill=yellow,circle,draw] at (4,-1) {}; \node(j)[fill=magenta,circle,draw] at (2,-1) {};
\node(l1)[] at  (-2,1) {$\tilde{\mcG}$:};
\draw[line width = 4mm, cyan] (a) -- (c) -- (d);
\draw[line width = 4mm, cyan] (b) -- (c) -- (j);
\draw[very thick] (a) -- (b) -- (c) -- (d) -- (e) -- (f) -- (g) -- (h) -- (f) -- (i) -- (j) -- (c) -- (a);

\fill[cyan] (12,2.5) circle (6mm); \fill[cyan] (14,2.5) circle (6mm); \fill[cyan] (13,0.5) circle (6mm); \fill[cyan] (14,-1.5) circle (6mm); \fill[cyan] (12,-1.5) circle (6mm);
\node(a)[fill=red,circle,draw] at (12,2.5) {}; \node(b)[fill=blue,circle,draw] at (14,2.5) {}; \node (c)[fill=green,circle,draw] at (13,0.5) {};
\node(d)[fill=yellow,circle,draw] at (14,-1.5) {}; \node(e)[fill= magenta,circle,draw] at (12,-1.5) {};
\draw[line width = 4mm, cyan] (a) -- (c) -- (d);
\draw[line width = 4mm, cyan] (b) -- (c) -- (e);
\draw[very thick] (a) -- (b) -- (c) -- (d) -- (e) -- (c) -- (a);
\node(l2)[] at  (11,1) {$\mcG$:};
\end{tikzpicture}
\end{center}

Two coverings $p_i: \tilde{\mcG}_i \to \mcG$, $i = 1,2$ are said to be isomorphic if there exists a graph isomorphism $\Phi: \tilde{\mcG}_1 \to \tilde{\mcG}_2$ such that the following diagram commutes:
\begin{center}
\begin{tikzpicture}[scale=0.8]
\node(a)[] at (-2,2) {$\tilde{\mcG_1}$}; \node(b)[] at (2,2) {$\tilde{\mcG_2}$}; \node (c)[] at (0,0) {$\mcG$};
\path (a) edge[thick, -stealth] node[anchor=north,above]{$\Phi$} (b);
\path (a) edge[thick, -stealth] node[anchor=west,below]{$p_1$ \ } (c);
\path (b) edge[thick, -stealth] node[anchor=east,below]{\ \ $p_2$} (c);
\end{tikzpicture}
\end{center}
The right hand side of Equation \ref{the_formula} is equal to the number of isomorphism classes of $m$-fold coverings of $\mcG$ with $\be(\mcG)=r-1$ (see \cite{KwakLee}).

In light of this graphical interpretation of Equation \ref{the_formula}, one anticipates a bijective correspondence between finite graph coverings and measurements of quantum entanglement. Indeed, such a correspondence exists, which we illustrate in the next section.

\subsection{The correspondence}

If $V$ is a complex vector space then we denote\footnote{Here we are viewing $V$ as a \emph{complex} space rather than a \emph{real} space, as we do in defining $\mcPr(V)$.} the complex valued polynomial functions on $V$ by $\mcP(V)$.  Suppose that a compact Lie group, $K$, acts $\bbC$-linearly on $V$.  The $K$-action on $V$ gives rise to an action on $\mcP(V)$ by $k \cdot f(v) = f(k^{-1}v)$ for $k \in K$, $f \in \mcP(V)$ and $v \in V$.   Both $\mcP(V)$ and $\mcPr(V)$ are complex vector spaces with a natural gradation by degree.  As a graded representation, $\mcPr(V) \cong \mcP\left(V \oplus \overline{V} \right)$, where $\overline{V}$ denotes the complex vector space with the opposite complex structure (see \cites{MW02}).  Let $V^*$ refer to the representation on the complex valued linear functionals on $V$ defined by $(k \cdot \la)(v) = \la(k^{-1} v)$ for $v \in V$, $\la \in V^*$, and $k \in K$. As a representation of $K$, $\overline V$ is equivalent to $V^*$.

In what is to follow, we will complexify the compact group, $K$, to a complex reductive linear algebraic group.  All representations of $G$ will be assumed to be regular.  That is, the matrix coefficients are regular functions on the underlying affine variety $G$.  An irreducible regular representation restricts to an irreducible complex representation of $K$.  Furthermore, since $K$ is Zariski dense in $G$, regular representations of $G$ (and hence $G$-invariants) are determined on $K$.  Note that, $G=\GL(n)$ when $K=U(n)$.

We now specialize to $V = V(\bfn) = \bbC^{n_1} \ot \cdots \ot \bbC^{n_r}$, and set up notation for the coordinates in $V$ and $V^*$.  For positive integers $k$ and $n$, let $\mbox{Mat}_{n,k}$ denote the vector space of $n \times k$ complex matrices.  Let $E^i_j \in \mbox{Mat}_{n,k}$ denote the matrix with entry in row $i$ and column $j$ equal to 1 and all other entries 0.  The group of $n \times n$ invertible matrices with complex number entries will be denoted by $\GL(n)$.  This group acts on $\mbox{Mat}_{n,k}$ by multiplication on the left.  We identify $\bbC^n = \mbox{Mat}_{n,1}$, which has a distinguished ordered basis consisting of $e_i = E^i_1 \in \mbox{Mat}_{n,1}$ for $i=1, \cdots, n$.

In the case of $G = \GL(n)$ we will identify $\dual{\bbC^n}$ with the representation on $\mbox{Mat}_{1,n}$ defined by the action $g \cdot v = v g^{-1}$ for $v \in \mbox{Mat}_{1,n}$ and $g \in \GL(n)$.  Set $e^i = E^1_i \in \mbox{Mat}_{1,n}$ for $i = 1, \cdots, n$. Then, $(e^1, \cdots, e^n)$ is an ordered basis for $\dual{\bbC^n}$, dual to $(e_1, \cdots, e_n)$.

Arbitrary tensors in $\V$ and $\V^*$ are of the form
\[ \sum x^{i_1 \cdots i_{r}} e_{i_1} \otimes \cdots \otimes e_{i_{r}} \in \V,\]
and
\[ \sum y_{i_1 \cdots i_{r}} e^{i_1} \otimes \cdots \otimes e^{i_{r}} \in \V^*, \]
where $x^{i_1 \cdots i_{r}}$ and $y_{i_1 \cdots i_{r}}$ are complex scalars.  We may view the variables $x^{i_1 \cdots i_{r}}$ and $y_{i_1 \cdots i_{r}}$ as degree 1 polynomial functions in $\mcPr(\V)$, where $y_{i_1 \cdots i_{r}}$ are the complex conjugates of $x^{i_1 \cdots i_{r}}$.

Let $\mcG$ be a connected simple graph with $\be(\mcG)=r-1$.  In \cite{graph_enum} the isomorphism classes of $m$-fold covers of $\mcG$ are parameterized by the orbits in $S_m^{r-1} = S_m \times \cdots \times S_m$ ($r-1$ factors) under the conjugation action of $S_m$.
Thus, one expects to form a basis\footnote{In general, one obtains a spanning set for the invariants.  However, if $n_i \geq m$ for all $i$ then we have a basis for the degree $2m$ invariants.} element of the space of degree $2m$ invariants from a choice, up to conjugation, of $r-1$ permutations.  Let
\[
[\sigma_1, \cdots, \sigma_{r-1}] = \left\{ \tau (\sigma_1, \cdots, \sigma_{r-1}) \tau^{-1}: \tau \in S_m \right\}
\]
be such a choice.  We present now an invariant associated with $[\sigma_1, \cdots, \sigma_{r-1}]$.

We define $f_{[\sigma_1, \cdots, \sigma_{r-1}]}$ as the sum over the indices
\[
I_1 = (i^{(1)}_1 i^{(1)}_2 \cdots i^{(1)}_{r}),
\cdots,
I_m = (i^{(m)}_1 i^{(m)}_2 \cdots i^{(m)}_{r})
\]
where $1 \leq i^{(j)}_k \leq n_k$ (with $j=1, \cdots, m$) of
\[
x^{ I_1 } \cdots x^{ I_m }
y_{ i^{(\sigma_1(1))}_{1} \cdots i^{(\sigma_{r-1}(1))}_{r-1} i^{(1)}_{r} }
\cdots
y_{ i^{(\sigma_1(m))}_{1} \cdots i^{(\sigma_{r-1}(m))}_{r-1} i^{(m)}_{r}. }
\]

We simultaneously parameterize the degree $2m$ polynomial invariants and $m$-fold coverings of simple connected graphs in the following way. Given a double coset $S^r_m$, applying $\theta$, one obtains an ($r$-1)-tuple of permutations $(\sigma_1, \cdots, \sigma_{r-1})$.
Combinatorially, we can encode the $S_m$-orbit of ($\sigma_1, \cdots, \sigma_{r-1}$) under the simultaneous conjugation action by coloring each permutation.  This action ``forgets'' the labels of the domain and range of each permutation.  The resulting combinatorial data takes the form of an unlabeled directed graph with edges colored by $r-1$ colors.

We will now exhibit this process in the case $r=m=3$. Consider $S_3(\tilde{\sigma_1},\tilde{\sigma_2},\tilde{\sigma_3})S_3$ where $\tilde{\sigma_1}=(1 \ 3 \ 2), \tilde{\sigma_2} = (2 \ 3), \tilde{\sigma_3} = (1 \ 3)$. Then 
\[ \theta(\tilde{\sigma_1},\tilde{\sigma_2},\tilde{\sigma_3}) = (\tilde{\sigma_1}\tilde{\sigma_3}^{-1},\tilde{\sigma_2}\tilde{\sigma_3}^{-1}) = (\sigma_1,\sigma_2) = ((1 \ 2),(1 \ 2 \ 3)). \]

\begin{center}
\begin{tikzpicture}[scale=.65,decoration={markings, mark=at position 0.55 with {\arrow[black]{stealth};}}]
	 \node (a1) {1};
	 \node (b1) [below=0.2cm of a1]{2};
	 \node (c1) [below=0.2cm of b1]{3};
	
	 \node (e1) [right=0.8cm of a1]{1};
	 \node (f1) [right=0.8cm of b1]{2};
	 \node (g1) [right=0.8cm of c1]{3};
	
	 \node (i1) [right=0.8cm of e1]{1};
	 \node (j1) [right=0.8cm of f1]{2};
	 \node (k1) [right=0.8cm of g1]{3};
	
	 \draw[thick, blue, -stealth, shorten >=-2pt] (a1) to node[auto] {} (f1);
	 \draw[thick, blue, -stealth, shorten >=-2pt] (b1) to node[auto] {} (e1);
	 \draw[thick, blue, -stealth, shorten >=-2pt] (c1) to node[auto] {} (g1);

	 \draw[thick, red, -stealth, shorten >=-2pt] (e1) to node[auto] {} (j1);
	 \draw[thick, red, -stealth, shorten >=-2pt] (f1) to node[auto] {} (k1);
	 \draw[thick, red, -stealth, shorten >=-2pt] (g1) to node[auto] {} (i1);

	 \node (d1) [right=0.2cm of a1]{};
	 \node (d2) [right=0.2cm of e1]{};
	 \node (s) [above=0.1cm of e1]{$((1 \ 2),(1 \ 2 \ 3))$};
	
	 \node (a2)  [right=2.75cm of i1][fill,circle,inner sep=1.2pt] {};
	 \node (b2) [below=0.6cm of a2][fill,circle,inner sep=1.2pt] {};
	 \node (c2) [below=0.6cm of b2][fill,circle,inner sep=1.2pt] {};

	 \node (e2) [right=0.4cm of a2][fill,circle,inner sep=1.2pt] {};
	 \node (f2) [right=0.4cm of b2][fill,circle,inner sep=1.2pt] {};
	 \node (g2) [right=0.4cm of c2][fill,circle,inner sep=1.2pt] {};

	 \node (i2) [right=0.4cm of e2][fill,circle,inner sep=1.2pt] {};
	 \node (j2) [right=0.4cm of f2][fill,circle,inner sep=1.2pt] {};
	 \node (k2) [right=0.4cm of g2][fill,circle,inner sep=1.2pt] {};
	 \node (t) [above=0.35cm of e2]{$\tau((1 \ 2),(1 \ 2 \ 3))\tau^{-1}$};
	
	 \draw[thick, blue, -stealth] (a2) .. controls +(150:1) and +(180:1) .. (f2);
	 \draw[thick, blue, -stealth] (b2) .. controls +(210:1) and +(180:1) .. (e2);
	 \draw[thick, blue, -stealth] (c2)  .. controls +(130:1) and +(130:1) .. (g2);

	 \draw[thick, red, -stealth] (e2) .. controls +(110:1) and +(90:1) .. (j2);
	 \draw[thick, red, -stealth] (f2) .. controls +(110:1) and +(90:1) .. (k2);
	 \draw[thick, red, -stealth] (g2).. controls +(20:2) and +(40:1) .. (i2);

	 \node (a3) [right=3.5cm of i2][fill,circle,inner sep=1.2pt] {};
	 \node (b3) [below=0.6cm of a3][fill,circle,inner sep=1.2pt] {};
	 \node (c3) [below=0.6cm of b3][fill,circle,inner sep=1.2pt] {};
	 \node (u) [above=0.85cm of a3]{3-fold Cover of};
	 \node (v) [above=0.35cm of a3]{ a "Figure 8"};

	 \draw[thick, blue, -stealth] (a3) .. controls +(200:1) and +(170:.5) .. (b3);
	 \draw[thick, blue, -stealth] (b3) .. controls +(220:1.5) and +(140:1.25) .. (a3);
	 \draw[thick, blue, -stealth] (c3)  .. controls +(180:1) and +(90:1) .. (c3);

	 \draw[thick, red, -stealth] (a3) .. controls +(340:1) and +(10:0.5) .. (b3);
	 \draw[thick, red, -stealth] (b3) .. controls +(340:1) and +(10:0.5) .. (c3);
	 \draw[thick, red, -stealth] (c3).. controls +(320:2) and +(40:1.5) .. (a3);

	\end{tikzpicture}
\end{center}

Using the invariant defined above, we have $f_{[\sigma_1, \sigma_{2}]}$ is the sum of terms of the form
\begin{align*}
x^{i^{(1)}_1i^{(1)}_2i^{(1)}_3} & x^{i^{(2)}_1i^{(2)}_2i^{(2)}_3}x^{i^{(3)}_1i^{(3)}_2i^{(3)}_3}
					 y_{ i^{(\sigma_1(1))}_{1} i^{(\sigma_{2}(1))}_{2} i^{(1)}_{3} }
					 y_{ i^{(\sigma_1(2))}_{1} i^{(\sigma_{2}(2))}_{2} i^{(2)}_{3} }
					 y_{ i^{(\sigma_1(3))}_{1} i^{(\sigma_{2}(3))}_{2} i^{(3)}_{3} } \\
				      &= x^{i^{(1)}_1i^{(1)}_2i^{(1)}_3}x^{i^{(2)}_1i^{(2)}_2i^{(2)}_3}x^{i^{(3)}_1i^{(3)}_2i^{(3)}_3}
					 y_{ i^{(2)}_{1} i^{(2)}_{2} i^{(1)}_{3} }
					 y_{ i^{(1)}_{1} i^{(3)}_{2} i^{(2)}_{3} }
					 y_{ i^{(3)}_{1} i^{(1)}_{2} i^{(3)}_{3} }
\end{align*}

All possible diagrams for the $r=m=3$ case are shown below:

\bigskip

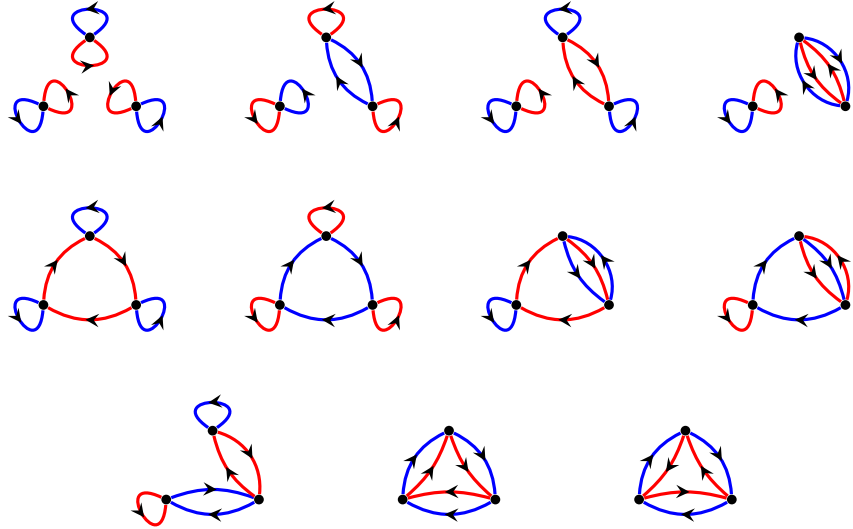
\begin{center}
\begin{tikzpicture}[scale=.85,decoration={markings, mark=at position 0.55 with {\arrow[black]{stealth};}}]
 \node (a1) [fill,circle,inner sep=1.3pt] {};
 \node (o1) [right=.4cm of a1] {}; 
 \node (b1) [above=.7cm of o1][fill,circle,inner sep=1.3pt] {};
 \node (c1) [right=.4cm of o1][fill,circle,inner sep=1.3pt] {};
 \node (a2) [right=3cm of a1][fill,circle,inner sep=1.3pt] {};
 \node (b2) [right=3cm of b1][fill,circle,inner sep=1.3pt] {};
 \node (c2) [right=3cm of c1][fill,circle,inner sep=1.3pt] {};
 \node (a3) [right=3cm of a2][fill,circle,inner sep=1.3pt] {};
 \node (b3) [right=3cm of b2][fill,circle,inner sep=1.3pt] {};
 \node (c3) [right=3cm of c2][fill,circle,inner sep=1.3pt] {};
 \node (a4) [right=3cm of a3][fill,circle,inner sep=1.3pt] {};
 \node (b4) [right=3cm of b3][fill,circle,inner sep=1.3pt] {};
 \node (c4) [right=3cm of c3][fill,circle,inner sep=1.3pt] {};
 \node (a5) [below=2.5cm of a1][fill,circle,inner sep=1.3pt] {};
 \node (b5) [below=2.5cm of b1][fill,circle,inner sep=1.3pt] {};
 \node (c5) [below=2.5cm of c1][fill,circle,inner sep=1.3pt] {};
 \node (d1) [right=.8cm of c5] {}; 
 \node (a6) [below=2.5cm of a2][fill,circle,inner sep=1.3pt] {};
 \node (b6) [below=2.5cm of b2][fill,circle,inner sep=1.3pt] {};
 \node (c6) [below=2.5cm of c2][fill,circle,inner sep=1.3pt] {};
 \node (a7) [below=2.5cm of a3][fill,circle,inner sep=1.3pt] {};
 \node (b7) [below=2.5cm of b3][fill,circle,inner sep=1.3pt] {};
 \node (c7) [below=2.5cm of c3][fill,circle,inner sep=1.3pt] {};
 \node (a8) [below=2.5cm of a4][fill,circle,inner sep=1.3pt] {};
 \node (b8) [below=2.5cm of b4][fill,circle,inner sep=1.3pt] {};
 \node (c8) [below=2.5cm of c4][fill,circle,inner sep=1.3pt] {};
 \node (o9) [below=2.3cm of d1] {}; 
 \node (a9) [left=.4cm of o9][fill,circle,inner sep=1.3pt] {};
 \node (b9) [above=.7cm of o9][fill,circle,inner sep=1.3pt] {};
 \node (c9) [right=.4cm of o9][fill,circle,inner sep=1.3pt] {};
 \node (a10) [right=3cm of a9][fill,circle,inner sep=1.3pt] {};
 \node (b10) [right=3cm of b9][fill,circle,inner sep=1.3pt] {};
 \node (c10) [right=3cm of c9][fill,circle,inner sep=1.3pt] {};
 \node (a11) [right=3cm of a10][fill,circle,inner sep=1.3pt] {};
 \node (b11) [right=3cm of b10][fill,circle,inner sep=1.3pt] {};
 \node (c11) [right=3cm of c10][fill,circle,inner sep=1.3pt] {};

 \draw[very thick,blue,rotate=120,postaction=decorate] 	(a1) .. controls +(35:1cm) and +(145:1cm) .. (a1); 
 \draw[very thick,blue,postaction=decorate] 						(b1) .. controls +(35:1cm) and +(145:1cm) .. (b1);
 \draw[very thick,blue,rotate=240,postaction=decorate] 	(c1) .. controls +(35:1cm) and +(145:1cm) .. (c1);
 \draw[very thick,red,rotate=120,postaction=decorate] 	(a1) .. controls +(-145:1cm) and +(-35:1cm) .. (a1);
 \draw[very thick,red,postaction=decorate] 							(b1) .. controls +(-145:1cm) and +(-35:1cm) .. (b1);
 \draw[very thick,red,rotate=240,postaction=decorate] 	(c1) .. controls +(-145:1cm) and +(-35:1cm) .. (c1);

 \draw[very thick,blue,rotate=120,postaction=decorate] 	(a2) .. controls +(-145:1cm) and +(-35:1cm) .. (a2); 
 \draw[very thick,blue,postaction=decorate] 						(b2) to[bend left=30] node[auto] {} (c2);
 \draw[very thick,blue,postaction=decorate] 						(c2) to[bend left=30] node[auto] {} (b2);
 \draw[very thick,red,rotate=120,postaction=decorate] 	(a2) .. controls +(35:1cm) and +(145:1cm) .. (a2);
 \draw[very thick,red,postaction=decorate] 							(b2) .. controls +(35:1cm) and +(145:1cm) .. (b2);
 \draw[very thick,red,rotate=240,postaction=decorate] 	(c2) .. controls +(35:1cm) and +(145:1cm) .. (c2);

 \draw[very thick,blue,rotate=120,postaction=decorate] 	(a3) .. controls +(35:1cm) and +(145:1cm) .. (a3); 
 \draw[very thick,blue,postaction=decorate] 						(b3) .. controls +(35:1cm) and +(145:1cm) .. (b3);
 \draw[very thick,blue,rotate=240,postaction=decorate]	(c3) .. controls +(35:1cm) and +(145:1cm) .. (c3);
 \draw[very thick,red,rotate=120,postaction=decorate] 	(a3) .. controls +(-145:1cm) and +(-35:1cm) .. (a3);
 \draw[very thick,red,postaction=decorate] 							(b3) to[bend left=30] node[auto] {} (c3);
 \draw[very thick,red,postaction=decorate] 							(c3) to[bend left=30] node[auto] {} (b3);

 \draw[very thick,blue,rotate=120,postaction=decorate] 	(a4) .. controls +(35:1cm) and +(145:1cm) .. (a4); 
 \draw[very thick,blue,postaction=decorate] 						(b4) to[bend left=50] node[auto] {} (c4);
 \draw[very thick,blue,postaction=decorate] 						(c4) to[bend left=50] node[auto] {} (b4);
 \draw[very thick,red,rotate=120,postaction=decorate] 	(a4) .. controls +(-145:1cm) and +(-35:1cm) .. (a4);
 \draw[very thick,red,postaction=decorate] 							(b4) to[bend right=20] node[auto] {} (c4);
 \draw[very thick,red,postaction=decorate] 							(c4) to[bend right=20] node[auto] {} (b4);

 \draw[very thick,blue,rotate=120,postaction=decorate] 	(a5) .. controls +(35:1cm) and +(145:1cm) .. (a5); 
 \draw[very thick,blue,postaction=decorate] 						(b5) .. controls +(35:1cm) and +(145:1cm) .. (b5);
 \draw[very thick,blue,rotate=240,postaction=decorate]	(c5) .. controls +(35:1cm) and +(145:1cm) .. (c5);
 \draw[very thick,red,postaction=decorate] 							(a5) to[bend left=30] node[auto] {} (b5);
 \draw[very thick,red,postaction=decorate]							(b5) to[bend left=30] node[auto] {} (c5);
 \draw[very thick,red,postaction=decorate] 							(c5) to[bend left=30] node[auto] {} (a5);

 \draw[very thick,blue,postaction=decorate] 						(a6) to[bend left=30] node[auto] {} (b6); 
 \draw[very thick,blue,postaction=decorate] 						(b6) to[bend left=30] node[auto] {} (c6);
 \draw[very thick,blue,postaction=decorate] 						(c6) to[bend left=30] node[auto] {} (a6);
 \draw[very thick,red,rotate=120,postaction=decorate] 	(a6) .. controls +(35:1cm) and +(145:1cm) .. (a6);
 \draw[very thick,red,postaction=decorate] 							(b6) .. controls +(35:1cm) and +(145:1cm) .. (b6);
 \draw[very thick,red,rotate=240,postaction=decorate] 	(c6) .. controls +(35:1cm) and +(145:1cm) .. (c6);

 \draw[very thick,blue,rotate=120,postaction=decorate] 	(a7) .. controls +(35:1cm) and +(145:1cm) .. (a7); 
 \draw[very thick,blue,postaction=decorate] 						(b7) to[bend right=20] node[auto] {} (c7);
 \draw[very thick,blue,postaction=decorate] 						(c7) to[bend right=50] node[auto] {} (b7);
 \draw[very thick,red,postaction=decorate] 							(a7) to[bend left=30] node[auto] {} (b7);
 \draw[very thick,red,postaction=decorate] 							(b7) to[bend left=20] node[auto] {} (c7);
 \draw[very thick,red,postaction=decorate] 							(c7) to[bend left=30] node[auto] {} (a7);

 \draw[very thick,blue,postaction=decorate] 						(a8) to[bend left=30] node[auto] {} (b8); 
 \draw[very thick,blue,postaction=decorate] 						(b8) to[bend left=20] node[auto] {} (c8);
 \draw[very thick,blue,postaction=decorate] 						(c8) to[bend left=30] node[auto] {} (a8);
 \draw[very thick,red,rotate=120,postaction=decorate] 	(a8) .. controls +(35:1cm) and +(145:1cm) .. (a8);
 \draw[very thick,red,postaction=decorate] 							(b8) to[bend right=20] node[auto] {} (c8);
 \draw[very thick,red,postaction=decorate] 							(c8) to[bend right=50] node[auto] {} (b8);

 \draw[very thick,blue,postaction=decorate] 						(a9) to[bend left=20] node[auto] {} (c9); 
 \draw[very thick,blue,postaction=decorate] 						(b9) .. controls +(35:1cm) and +(145:1cm) .. (b9);
 \draw[very thick,blue,postaction=decorate] 						(c9) to[bend left=30] node[auto] {} (a9);
 \draw[very thick,red,rotate=120,postaction=decorate] 	(a9) .. controls +(35:1cm) and +(145:1cm) .. (a9);
 \draw[very thick,red,postaction=decorate] 							(b9) to[bend left=40] node[auto] {} (c9);
 \draw[very thick,red,postaction=decorate] 							(c9) to[bend left=20] node[auto] {} (b9);

 \draw[very thick,blue,postaction=decorate] 	(a10) to[bend left=30] node[auto] {} (b10); 
 \draw[very thick,blue,postaction=decorate] 	(b10) to[bend left=30] node[auto] {} (c10);
 \draw[very thick,blue,postaction=decorate] 	(c10) to[bend left=30] node[auto] {} (a10);
 \draw[very thick,red,postaction=decorate] 		(c10) to[bend right=15] node[auto] {} (a10);
 \draw[very thick,red,postaction=decorate] 		(a10) to[bend right=15] node[auto] {} (b10);
 \draw[very thick,red,postaction=decorate] 		(b10) to[bend right=15] node[auto] {} (c10);

 \draw[very thick,blue,postaction=decorate] 	(a11) to[bend left=30] node[auto] {} (b11); 
 \draw[very thick,blue,postaction=decorate] 	(b11) to[bend left=30] node[auto] {} (c11);
 \draw[very thick,blue,postaction=decorate] 	(c11) to[bend left=30] node[auto] {} (a11);
 \draw[very thick,red,postaction=decorate] 		(a11) to[bend left=15] node[auto] {} (c11);
 \draw[very thick,red,postaction=decorate] 		(b11) to[bend left=15] node[auto] {} (a11);
 \draw[very thick,red,postaction=decorate] 		(c11) to[bend left=15] node[auto] {} (b11);
\end{tikzpicture}
\end{center}

Each coloring of the directed graphs corresponds to an isomorphism class of $m$-fold covering of a connected simple graph $G$ with $\be(G) = r-1$ (see \cite{graph_enum}).  We illustrate this correspondence for $m=2$, and $r=2,3,4$.  In the following table, the simple graph is homotopic to a bouquet of loops (on the left) and the possible graph coverings are on the right.  The colors and orientations determine the covering map.
The corresponding $K(\bfn)$-invariants are written out explicitly following the Einstein summation convention.

{\tiny

\begin{table}[h!t!]
    \centering
    \begin{tabular}{@{}c@{} | l}
\toprule
    $\beta(G) = 1$, $m=2$ & $r=2$, $m=2$
\tabularnewline
			\begin{tikzpicture}[scale=0.75,node distance = 2 cm,decoration={markings, mark=at position .55 with {\arrow[black]{stealth};}}]
 				\node (a)  {};
 				\node (b) [left of=a][fill,circle,inner sep=1.7pt] {};
 				\draw[very thick,blue,postaction=decorate,shorten >=-3pt] (b) .. controls +(90:1cm) and +(90:1cm) .. (a);
 				\draw[very thick,blue,shorten >=-3.5pt] (b) .. controls +(-90:1cm) and +(-90:1cm) .. (a);
				\end{tikzpicture}
      &
      \begin{tikzpicture}[scale=0.6,node distance = 1 cm,decoration={markings, mark=at position 0.55 with {\arrow[black]{stealth};}}]
 				\node (a) [fill,circle,inner sep=1.4pt] {};
 				\node (b) [below of=a][fill,circle,inner sep=1.4pt] {};
 				\node (1) [above left of=b] {(1)};
 				\node (c) [right of=a] {};
 				\node (d) [right of=b] {};
 				\node (e) [right of=c][fill,circle,inner sep=1.4pt] {};
 				\node (f) [right of=d][fill,circle,inner sep=1.4pt] {};
 				\node (2) [above left of=f] {(2)};

 				\draw[very thick,blue,postaction=decorate] (a) .. controls +(35:1cm) and +(145:1cm) .. (a); 
 				\draw[very thick,blue,postaction=decorate] (b) .. controls +(35:1cm) and +(145:1cm) .. (b);
 				\draw[very thick,blue,postaction=decorate] (e) to[bend right=60] node[auto] {} (f); 
 				\draw[very thick,blue,postaction=decorate] (f) to[bend right=60] node[auto] {} (e);
			\end{tikzpicture}
\tabularnewline
\multicolumn{2}{c}{ (1) $x^{i^{(1)}_1i^{(1)}_2}y_{i^{(1)}_1i^{(1)}_2}x^{i^{(2)}_1i^{(2)}_2}y_{i^{(2)}_1i^{(2)}_2}$ \hspace*{2cm}
	  (2) $x^{i^{(1)}_1i^{(1)}_2}y_{i^{(2)}_1i^{(1)}_2}x^{i^{(2)}_1i^{(2)}_2}y_{i^{(1)}_1i^{(2)}_2}$}
\tabularnewline\midrule
   $\beta(G) = 2$, $m=2$ & $r=3$, $m=2$
\tabularnewline
      \begin{tikzpicture}[scale=0.85,node distance = 1.6 cm,decoration={markings, mark=at position .5 with {\arrow[black]{stealth};}}]
 				\node (a) [fill,circle,inner sep=1.7pt] {};
 				\node (b) [left of=a] {};
 				\node (c) [right of=a] {};
 	      \draw[very thick,blue,postaction=decorate,shorten <=-3pt] (b) .. controls +(90:1cm) and +(135:1cm) .. (a);
 				\draw[very thick,blue,shorten <=-3.5pt]    (b) .. controls +(-90:1cm) and +(-135:1cm) .. (a);
 				\draw[very thick,postaction=decorate,red,shorten <=-3pt]    (c) .. controls +(90:1cm) and +(45:1cm) .. (a);
 				\draw[very thick,red,shorten <=-3.5pt] (c) .. controls +(-90:1cm) and +(-45:1cm) .. (a);
			\end{tikzpicture}
      &
      \begin{tikzpicture}[scale=0.6,node distance = 1 cm,decoration={markings, mark=at position 0.55 with {\arrow[black]{stealth};}}]
 				\node (a) [fill,circle,inner sep=1.4pt] {};
 				\node (b) [below of=a][fill,circle,inner sep=1.4pt] {};
 				\node (1) [above left of=b] {(1)};
 				\node (c) [right of=a] {};
 				\node (d) [right of=b] {};
 				\node (e) [right of=c][fill,circle,inner sep=1.4pt] {};
 				\node (f) [right of=d][fill,circle,inner sep=1.4pt] {};
 				\node (2) [above left of=f] {(2)};
 				\node (g) [right of=e] {};
 				\node (h) [right of=f] {};
 				\node (i) [right of=g][fill,circle,inner sep=1.4pt] {};
 				\node (j) [right of=h][fill,circle,inner sep=1.4pt] {};
 				\node (3) [above left of=j] {(3)};
 				\node (k) [right of=i] {};
 				\node (l) [right of=j] {};
 				\node (m) [right of=k][fill,circle,inner sep=1.4pt] {};
 				\node (n) [right of=l][fill,circle,inner sep=1.4pt] {};
 				\node (4) [above left of=n] {(4)};

 				\draw[very thick,blue,postaction=decorate] (a) .. controls +(-145:1cm) and +(-35:1cm) .. (a); 
 				\draw[very thick,blue,postaction=decorate] (b) .. controls +(-145:1cm) and +(-35:1cm) .. (b);
 				\draw[very thick,red,postaction=decorate] (a) .. controls +(145:1cm) and +(35:1cm) .. (a);
 				\draw[very thick,red,postaction=decorate] (b) .. controls +(145:1cm) and +(35:1cm) .. (b);

 				\draw[very thick,blue,postaction=decorate] (f) to[bend left=60] node[auto] {} (e); 
 				\draw[very thick,blue,postaction=decorate] (e) to[bend left=60] node[auto] {} (f);
 				\draw[very thick,red,postaction=decorate] (e) .. controls +(35:1cm) and +(145:1cm) .. (e);
 				\draw[very thick,red,postaction=decorate] (f) .. controls +(-145:1cm) and +(-35:1cm) .. (f);

 				\draw[very thick,blue,postaction=decorate] (i) .. controls +(35:1cm) and +(145:1cm) .. (i); 
 				\draw[very thick,blue,postaction=decorate] (j) .. controls +(-145:1cm) and +(-35:1cm) .. (j);
 				\draw[very thick,red,postaction=decorate] (i) to[bend left=60] node[auto] {} (j);
 				\draw[very thick,red,postaction=decorate] (j) to[bend left=60] node[auto] {} (i);

 				\draw[very thick,blue,postaction=decorate] (m) to[bend left=70] node[auto] {} (n); 
 				\draw[very thick,blue,postaction=decorate] (n) to[bend left=70] node[auto] {} (m);
 				\draw[very thick,red,postaction=decorate] (m) to[bend right=40] node[auto] {} (n);
 				\draw[very thick,red,postaction=decorate] (n) to[bend right=40] node[auto] {} (m);
			\end{tikzpicture}
\tabularnewline
\multicolumn{2}{c}{ (1) $x^{i^{(1)}_1i^{(1)}_2i^{(1)}_3}y_{i^{(1)}_1i^{(1)}_2i^{(1)}_3}x^{i^{(2)}_1i^{(2)}_2i^{(2)}_3}y_{i^{(2)}_1i^{(2)}_2i^{(2)}_3}$ \hspace*{1cm}
	  (3) $x^{i^{(1)}_1i^{(1)}_2i^{(1)}_3}y_{i^{(1)}_1i^{(2)}_2i^{(1)}_3}x^{i^{(2)}_1i^{(2)}_2i^{(2)}_3}y_{i^{(2)}_1i^{(1)}_2i^{(2)}_3}$}
 \tabularnewline
\multicolumn{2}{c}{  (2) $x^{i^{(1)}_1i^{(1)}_2i^{(1)}_3}y_{i^{(2)}_1i^{(1)}_2i^{(1)}_3}x^{i^{(2)}_1i^{(2)}_2i^{(2)}_3}y_{i^{(1)}_1i^{(2)}_2i^{(2)}_3}$ \hspace*{1cm}
	  (4) $x^{i^{(1)}_1i^{(1)}_2i^{(1)}_3}y_{i^{(2)}_1i^{(2)}_2i^{(1)}_3}x^{i^{(2)}_1i^{(2)}_2i^{(2)}_3}y_{i^{(1)}_1i^{(1)}_2i^{(2)}_3}$}
\tabularnewline\midrule
   $\beta(G) = 3$, $m=2$ & $r=4$, $m=2$
\tabularnewline
			\begin{tikzpicture}[scale=0.8,node distance = 1.8 cm,decoration={markings, mark=at position 0.4 with {\arrow[black]{stealth};}}]
 				\node (a) [fill,circle,inner sep=1.7pt] {};
 				\node (o1) [left=1.35cm of a] {};
 				\node (b) [above=1cm of o1] {};
 				\node (o2) [right=1.35cm of a] {};
 				\node (c) [above=1cm of o2] {};
 				\node (d) [below of=a] {};
 				\draw[very thick,blue,postaction=decorate,shorten <=-4.5pt] 					 (b) .. controls +(50:1cm) and +(105:1cm) .. (a);
 				\draw[very thick,blue,shorten <=-4.5pt] 														 (b) .. controls +(230:1cm) and +(195:1cm) .. (a);
 				\draw[very thick,red,shorten <=-4.5pt] 																 (c) .. controls +(130:1cm) and +(75:1cm) .. (a);
 				\draw[very thick,red,postaction=decorate,shorten <=-4.5pt] 					 (c) .. controls +(310:1cm) and +(345:1cm) .. (a);
 				\draw[very thick,green!70!black,postaction=decorate,shorten <=-3.5pt] (d) .. controls +(180:1cm) and +(225:1cm) .. (a);
 				\draw[very thick,green!70!black,shorten <=-3pt] 									 (d) .. controls +(0:1cm) and +(315:1cm) .. (a);
			\end{tikzpicture}
      &
      \centering
			\begin{tikzpicture}[scale=0.6,node distance = 1 cm,decoration={markings, mark=at position 0.55 with {\arrow[black]{stealth};}}]
 				\node (a) [fill,circle,inner sep=1.4pt] {};
 				\node (b) [below of=a][fill,circle,inner sep=1.4pt] {};
 				\node (1) [above left of=b] {(1)};
 				\node (c) [right of=a] {};
 				\node (d) [right of=b] {};
 				\node (e) [right of=c][fill,circle,inner sep=1.4pt] {};
 				\node (f) [right of=d][fill,circle,inner sep=1.4pt] {};
 				\node (2) [above left of=f] {(2)};
 				\node (g) [right of=e] {};
 				\node (h) [right of=f] {};
 				\node (i) [right of=g][fill,circle,inner sep=1.4pt] {};
 				\node (j) [right of=h][fill,circle,inner sep=1.4pt] {};
 				\node (3) [above left of=j] {(3)};
 				\node (k) [right of=i] {};
 				\node (l) [right of=j] {};
 				\node (m) [right of=k][fill,circle,inner sep=1.4pt] {};
 				\node (n) [right of=l][fill,circle,inner sep=1.4pt] {};
 				\node (4) [above left of=n] {(4)};
 				\node (z) [right of=n] {};
 				\node (o) [below of=d][fill,circle,inner sep=1.4pt] {};
 				\node (p) [below of=o][fill,circle,inner sep=1.4pt] {};
 				\node (5) [above left of=p] {(5)};
 				\node (q) [below of=h][fill,circle,inner sep=1.4pt] {};
 				\node (r) [below of=q][fill,circle,inner sep=1.4pt] {};
 				\node (6) [above left of=r] {(6)};
 				\node (s) [below of=l][fill,circle,inner sep=1.4pt] {};
 				\node (t) [below of=s][fill,circle,inner sep=1.4pt] {};
 				\node (7) [above left of=t] {(7)};
 				\node (u) [below of=z][fill,circle,inner sep=1.4pt] {};
 				\node (v) [below of=u][fill,circle,inner sep=1.4pt] {};
 				\node (8) [above left of=v] {(8)};
 				
 				\draw[very thick,blue,postaction=decorate] (a) .. controls +(95:1cm) and +(205:1cm) .. (a); 
 				\draw[very thick,blue,postaction=decorate] (b) .. controls +(95:1cm) and +(205:1cm) .. (b);
 				\draw[very thick,red,postaction=decorate] (a) .. controls +(335:1cm) and +(85:1cm) .. (a);
 				\draw[very thick,red,postaction=decorate] (b) .. controls +(335:1cm) and +(85:1cm) .. (b);
 				\draw[very thick,green!70!black,postaction=decorate] (a) .. controls +(215:1cm) and +(325:1cm) .. (a);
 				\draw[very thick,green!70!black,postaction=decorate] (b) .. controls +(215:1cm) and +(325:1cm) .. (b);

 				\draw[very thick,blue,postaction=decorate] (e) to[bend right=60] node[auto] {} (f); 
 				\draw[very thick,blue,postaction=decorate] (f) to[bend right=60] node[auto] {} (e);
 				\draw[very thick,red,postaction=decorate] (e) .. controls +(335:1cm) and +(85:1cm) .. (e);
 				\draw[very thick,red,postaction=decorate] (f) .. controls +(25:1cm) and +(-85:1cm) .. (f);
 				\draw[very thick,green!70!black,postaction=decorate] (e) .. controls +(95:1cm) and +(205:1cm) .. (e);
 				\draw[very thick,green!70!black,postaction=decorate] (f) .. controls +(265:1cm) and +(155:1cm) .. (f);

 				\draw[very thick,blue,postaction=decorate] (i) .. controls +(335:1cm) and +(85:1cm) .. (i); 
 				\draw[very thick,blue,postaction=decorate] (j) .. controls +(25:1cm) and +(-85:1cm) .. (j);
 				\draw[very thick,red,postaction=decorate] (i) to[bend right=60] node[auto] {} (j);
 				\draw[very thick,red,postaction=decorate] (j) to[bend right=60] node[auto] {} (i);
 				\draw[very thick,green!70!black,postaction=decorate] (i) .. controls +(95:1cm) and +(205:1cm) .. (i);
 				\draw[very thick,green!70!black,postaction=decorate] (j) .. controls +(265:1cm) and +(155:1cm) .. (j);

 				\draw[very thick,blue,postaction=decorate] (m) .. controls +(335:1cm) and +(85:1cm) .. (m); 
 				\draw[very thick,blue,postaction=decorate] (n) .. controls +(25:1cm) and +(-85:1cm) .. (n);
 				\draw[very thick,red,postaction=decorate] (m) .. controls +(95:1cm) and +(205:1cm) .. (m);
 				\draw[very thick,red,postaction=decorate] (n) .. controls +(265:1cm) and +(155:1cm) .. (n);
 				\draw[very thick,green!70!black,postaction=decorate] (m) to[bend right=60] node[auto] {} (n);
 				\draw[very thick,green!70!black,postaction=decorate] (n) to[bend right=60] node[auto] {} (m);
 				
 				\draw[very thick,blue,postaction=decorate] (o) to[bend left=40] node[auto] {} (p); 
 				\draw[very thick,blue,postaction=decorate] (p) to[bend left=40] node[auto] {} (o);
 				\draw[very thick,red,postaction=decorate] (p) to[bend right=70] node[auto] {} (o);
 				\draw[very thick,red,postaction=decorate] (o) to[bend right=70] node[auto] {} (p);
 				\draw[very thick,green!70!black,postaction=decorate] (o) .. controls +(35:1cm) and +(145:1cm) .. (o);
 				\draw[very thick,green!70!black,postaction=decorate] (p) .. controls +(-145:1cm) and +(-35:1cm) .. (p);

 				\draw[very thick,blue,postaction=decorate] (q) .. controls +(35:1cm) and +(145:1cm) .. (q); 
 				\draw[very thick,blue,postaction=decorate] (r) .. controls +(-145:1cm) and +(-35:1cm) .. (r);
 				\draw[very thick,red,postaction=decorate] (r) to[bend left=40] node[auto] {} (q);
 				\draw[very thick,red,postaction=decorate] (q) to[bend left=40] node[auto] {} (r);
 				\draw[very thick,green!70!black,postaction=decorate] (q) to[bend right=70] node[auto] {} (r);
 				\draw[very thick,green!70!black,postaction=decorate] (r) to[bend right=70] node[auto] {} (q);

 				\draw[very thick,blue,postaction=decorate] (s) to[bend right=70] node[auto] {} (t); 
 				\draw[very thick,blue,postaction=decorate] (t) to[bend right=70] node[auto] {} (s);
 				\draw[very thick,red,postaction=decorate] (s) .. controls +(35:1cm) and +(145:1cm) .. (s);
 				\draw[very thick,red,postaction=decorate] (t) .. controls +(-145:1cm) and +(-35:1cm) .. (t);
 				\draw[very thick,green!70!black,postaction=decorate] (t) to[bend left=40] node[auto] {} (s);
 				\draw[very thick,green!70!black,,postaction=decorate] (s) to[bend left=40] node[auto] {} (t);

 				\draw[very thick,blue,postaction=decorate] (u) to[bend left=90] node[auto] {} (v); 
 				\draw[very thick,blue,postaction=decorate] (v) to[bend left=90] node[auto] {} (u);
 				\draw[very thick,red,postaction=decorate] (v) to[bend right=50] node[auto] {} (u);
 				\draw[very thick,red,postaction=decorate] (u) to[bend right=50] node[auto] {} (v);
 				\draw[very thick,green!70!black,postaction=decorate] (u) to[bend left=25] node[auto] {} (v);
 				\draw[very thick,green!70!black,postaction=decorate] (v) to[bend left=25] node[auto] {} (u);
		\end{tikzpicture}
\tabularnewline
 \tabularnewline
\multicolumn{2}{c}{
(1) $x^{i^{(1)}_1i^{(1)}_2i^{(1)}_3i^{(1)}_4}y_{i^{(1)}_1i^{(1)}_2i^{(1)}_3i^{(1)}_4}x^{i^{(2)}_1i^{(2)}_2i^{(2)}_3i^{(2)}_4}y_{i^{(2)}_1i^{(2)}_2i^{(2)}_3i^{(2)}_4}$ }
	
 \tabularnewline
\multicolumn{2}{c}{ (2) $x^{i^{(1)}_1i^{(1)}_2i^{(1)}_3i^{(1)}_4}y_{i^{(2)}_1i^{(1)}_2i^{(1)}_3i^{(1)}_4}x^{i^{(2)}_1i^{(2)}_2i^{(2)}_3i^{(2)}_4}y_{i^{(1)}_1i^{(2)}_2i^{(2)}_3i^{(2)}_4}$ }

 \tabularnewline
\multicolumn{2}{c}{ (3) $x^{i^{(1)}_1i^{(1)}_2i^{(1)}_3i^{(1)}_4}y_{i^{(1)}_1i^{(2)}_2i^{(1)}_3i^{(1)}_4}x^{i^{(2)}_1i^{(2)}_2i^{(2)}_3i^{(2)}_4}y_{i^{(2)}_1i^{(1)}_2i^{(2)}_3i^{(2)}_4}$ }

 \tabularnewline
\multicolumn{2}{c}{ (4) $x^{i^{(1)}_1i^{(1)}_2i^{(1)}_3i^{(1)}_4}y_{i^{(1)}_1i^{(1)}_2i^{(2)}_3i^{(1)}_4}x^{i^{(2)}_1i^{(2)}_2i^{(2)}_3i^{(2)}_4}y_{i^{(2)}_1i^{(2)}_2i^{(1)}_3i^{(2)}_4}$ }

 \tabularnewline
\multicolumn{2}{c}{ (5) $x^{i^{(1)}_1i^{(1)}_2i^{(1)}_3i^{(1)}_4}y_{i^{(2)}_1i^{(2)}_2i^{(1)}_3i^{(1)}_4}x^{i^{(2)}_1i^{(2)}_2i^{(2)}_3i^{(2)}_4}y_{i^{(1)}_1i^{(1)}_2i^{(2)}_3i^{(2)}_4}$ }

 \tabularnewline
\multicolumn{2}{c}{ (6) $x^{i^{(1)}_1i^{(1)}_2i^{(1)}_3i^{(1)}_4}y_{i^{(1)}_1i^{(2)}_2i^{(2)}_3i^{(1)}_4}x^{i^{(2)}_1i^{(2)}_2i^{(2)}_3i^{(2)}_4}y_{i^{(2)}_1i^{(1)}_2i^{(1)}_3i^{(2)}_4}$ }

 \tabularnewline
\multicolumn{2}{c}{(7) $x^{i^{(1)}_1i^{(1)}_2i^{(1)}_3i^{(1)}_4}y_{i^{(2)}_1i^{(1)}_2i^{(2)}_3i^{(1)}_4}x^{i^{(2)}_1i^{(2)}_2i^{(2)}_3i^{(2)}_4}y_{i^{(1)}_1i^{(2)}_2i^{(1)}_3i^{(2)}_4}$ }

 \tabularnewline
\multicolumn{2}{c}{ (8) $x^{i^{(1)}_1i^{(1)}_2i^{(1)}_3i^{(1)}_4}y_{i^{(2)}_1i^{(2)}_2i^{(2)}_3i^{(1)}_4}x^{i^{(2)}_1i^{(2)}_2i^{(2)}_3i^{(2)}_4}y_{i^{(1)}_1i^{(1)}_2i^{(1)}_3i^{(2)}_4}$ }

\tabularnewline\bottomrule
    \end{tabular}
\end{table}
}

\bigskip
Finally, the fact that all invariants fall into this correspondence follows from
\begin{thm}\label{thm_basis_intro} For all $n_1, \cdots, n_r$ and $d = 2m$, we have
\[
    \mbox{Span} \{ f_{[\sigma_1, \cdots, \sigma_{r-1}]} : (\sigma_1, \cdots, \sigma_{r-1}) \in S_m^{r-1} \} = \mcPr^d(V(\bfn))^{K(\bfn)}.
\]
\end{thm}
\begin{proof}  Follows from Theorem \ref{thm_basis_poly} proved in Section \ref{sec_symtensor}.
\end{proof}

\newpage
\noindent{\bf Acknowledgments}  We would like to thank Jan Draisma for valuable comments in an initial draft of this manuscript, including the observation that Theorem \ref{thm_basis_intro} follows from standard arguments in classical invariant theory.
We would also like to thank Nolan Wallach for many helpful comments.

\newpage
\section{Invariants in the tensor algebra}\label{sec_tensor}

The group $\GL(n)$ has the structure of a reductive linear algebraic group over the field $\bbC$.  This article concerns the regular representations of such groups, which are closed under the operations of direct sum, tensor product, and duality.  That is, in general, if $V_1$ and $V_2$ are regular representations of a linear algebraic group $G$ then $V_1 \oplus V_2$ is a regular representation of $G$ defined by $g \cdot (v_1,v_2) = (g v_1, g v_2)$ and $V_1 \ot V_2$ is a regular representation of $G$ defined by $g \cdot (v_1 \ot v_2) = (g v_1) \ot (g v_2)$ (for $g \in G$ and $v_i \in V_i$), and extending by linearity.  Also if $V_1$ (resp. $V_2$) is a regular representation of an algebraic group $G_1$ (resp. $G_2$) then $V_1 \ot V_2$ also denotes the representation of $G_1 \times G_2$ defined by $(g_1, g_2) \cdot (v_1 \ot v_2) = (g_1 v_1) \ot (g_2 v_2)$.  In the case where $G_1 = G_2 = G$ then both $G$ and $G \times G$ act on $V_1 \ot V_2$.  The latter will be referred to as the ``outer'' action of $G \times G$ whose restriction to the diagonal subgroup, $\{(g,g): g\in G\}$, is equivalent to the former action, referred to as the ``inner'' action of $G$.  Throughout we will be careful to distinguish between these two actions, when there is ambiguity.

Let ${\bf n} = (n_1, \cdots, n_r)$ denote an $r$-tuple of positive integers. Let $\V = \bbC^{n_1} \ot \bbC^{n_2} \ot \cdots \ot \bbC^{n_r}$, which is an irreducible representation of the group $\G = \GL(n_1) \times \cdots \times \GL(n_r)$, under the outer action defined by
\[
    (g_1, \cdots, g_r) \cdot (v_1 \ot \cdots \ot v_r) = (g_1 v_1) \ot \cdots \ot (g_r v_r),
\] where for all $i$, $g_i \in \GL(n_i)$, $v_i \in \bbC^{n_i}$, and extending by linearity.

For a non-negative integer $d$, let
    \[ \mcT{d} = \bot^d\left[ \V \oplus \V^* \right].\]
The group $\G$ acts on $\mcT{d}$ by the inner action on the $d$-fold tensors on $\V \oplus \dual{\V}$.
The goal of this section is to find a spanning set for the $\G$-invariants, $\mcT{d}^\G$.  We shall see shortly that, upon examination of the action of the center of $\G$, $\mcT{d}^{\G} = \{0\}$ for $d$ odd.  Thus, we assume $d = 2m$ for a non-negative integer $m$.

If $V$ is a vector space, we introduce the notation $V_0 = V$ and $V_1 = V^*$.
We let $\bfb = (b_1, b_2, \cdots, b_d)$ denote a $d$-tuple of zeros and ones.  Set $\T(\bfn, \bfb)=\bigotimes_{i=1}^d \V_{b_i}$.  We have
\[
    \mcT{d} = \bigoplus_\bfb \T(\bfn, \bfb).
\] where the sum is over all $d$-long $\{0,1\}$-sequences, $\bfb$.   This equality follows from the bi-linearity of the tensor product (ie:  $(A \oplus B)\ot(C \oplus D) = A \ot C \oplus A \ot D \oplus B \ot C \oplus B \ot D$).  Thus, the problem of finding a basis for $\mcT{d}^\G$ reduces to finding a basis for $\T(\bfn, \bfb)^\G$ for each $\bfb$.  The tensor factors in $\T(\bfn, \bfb)$ may be re-ordered by defining
\[
    T_n(\bfb) = \bigotimes_{i=1}^d \left(\bbC^{n}\right)_{b_i},
\] which is a representation of $\GL(n)$ wrt the inner action.  Then, permuting the tensor factors so that those involving $\bbC^{n_i}$ are grouped together defines an isomorphism,
\[
    \Phi^\bfb_\bfn:\T(\bfn, \bfb) \rightarrow T_{n_1}(\bfb) \ot \cdots \ot T_{n_r}(\bfb)
\] of $\G$-representations.  We will now
obtain a basis for the $\G$-invariants of $\bigotimes_{i=1}^r T_{n_i}(\bfb)$, for each $\bfb$, under the assumption that $n_i$ are large with respect to $d$.

The center, denoted $\mcZ$, of $\G$ is $\{ (x_1 I_{n_1}, \cdots, x_r I_{n_r}) | x_1, \cdots, x_r \in \bbC^\times \}$ where $I_k$ is the $k \times k$ identity matrix ($k \in \bbZ^+$). Let $|\bfb|_1 = \sum_i b_i$ denote the number of 1's in $\bfb$, while $|\bfb|_0 = d - |\bfb|_1$ is the number of 0's in $\bfb$.
\begin{lemma}\label{lemma_equal_01}
For any $\bfn$ and $\bfb$, we have $\T(\bfn, \bfb)^\G = \{ 0 \}$ if $|\bfb|_1 \neq |\bfb|_0$.
\end{lemma}
\begin{proof}
The center of $\GL(n_i)$, $\{x_i I_{n_i}| x_i \in \bbC^\times \}$, acts on $T_{n_i}(\bfb)$ by the scalar $x_i^{|\bfb|_0 - |\bfb|_1}$.
Therefore $\mcZ$ acts trivially on $\T(\bfn, \bfb)$ exactly when $|\bfb|_1 = |\bfb|_0$.  \end{proof}
As a consequence of Lemma \ref{lemma_equal_01}, we will assume from this point on that $d = 2m$ is even, and that any $\bfb$ is to have $m$ 1's and $m$ 0's.  The converse of Lemma \ref{lemma_equal_01} is also true which we address next.

Suppose that $\bfb = 0^m1^m$ where $1^m$ (resp. $0^m$) is a sequence of $m$ 1's (resp. 0's). Then, for all $n$ we have $T_{n}(\bfb) = \left( \bot^m \bbC^{n} \right) \ot \dual{ \bigotimes^m \bbC^{n} }$.  For each permutation $\sigma \in S_m$, define:
\[
    t_n(\sigma) = \sum_{\twoline{(j_1, \cdots, j_m)}{1 \leq j_k \leq n, \forall k}}
    (e_{j_1} \ot \cdots \ot e_{j_m}) \ot (e^{j_{\sigma(1)}} \ot \cdots \ot e^{j_{\sigma(m)}})
\]  (Recall that in the above sum, $e_k$, and $e^k$, are the ordered bases of $\bbC^n$ and $\dual{\bbC^n}$ respectively defined previously.)

\begin{thm*}[Schur-Weyl Duality]\label{SWD} Let $n$ and $m$ be positive integers.  The tensor product, $\bot^m \bbC^n$, is a representation of $\GL(n)$ under the inner action, and also a representation of $S_m$ as defined by permutation of the tensor factors.  Each of these actions generates the full commuting associative algebra action.  Thus, we obtain a surjective algebra homomorphism
\[
    \bbC[S_m] \rightarrow \mbox{End}_{\GL(n)} \left( \bot^m \bbC^n \right),
\] which is an isomorphism if and only if $n \geq m$.
\end{thm*}
\begin{proof} See \cite{GW} Section 4.2.4 and Chapter 9.  \end{proof}

\begin{prop}\label{prop_from_SchurWeyl} For all $n$, if $n \geq m$ then
$\{t_n(\sigma) | \sigma \in S_m \}$ is a basis for $T_{n}(0^m1^m)^{\GL(n)}$,
otherwise, outside of these inequalities, the above are a spanning set.
\end{prop}
\begin{proof}
Observe that $T_n(\bfb)^{\GL(n)} = \left[ \left( \bot^m \bbC^{n} \right) \ot \dual{ \bigotimes^m \bbC^{n} } \right]^{\GL(n)} \cong \End_{\GL(n)}( \bigotimes^m \bbC^{n} )$.  The result follows from Schur-Weyl duality.
\end{proof}
For an $r$-tuple of permutations $\bfs = (\sigma_1, \cdots, \sigma_r) \in S_m^r$ define: $t(\bfs) = t_{n_1}(\sigma_1) \ot \cdots \ot t_{n_r}(\sigma_r)$.
We now obtain
\begin{cor}\label{cor_01b_basis_bfn}
Given $m$ and $\bf n$, if for all $i$ with $1 \leq i \leq r$, we have $n_i \geq m$ then
\[
    \left\{  t(\bfs) | \bfs \in S_m^r \right\}
\] is a vector space basis for the $\G$-invariants in $T_{n_1}(\bfb) \ot \cdots \ot T_{n_r}(\bfb)$ in the case when $\bfb = 0^m1^m$.  Otherwise, outside of these inequalities, the above are a spanning set.

\end{cor}
\begin{proof}  Given a finite collection of finite dimensional vector spaces $V_1, \cdots, V_r$, if $B_i$ is a basis for $V_i$ then $\{ v_1 \ot \cdots \ot v_r | v_j \; \in B_j \mbox{ for all } 1\leq j \leq r \}$ is a basis for $\bigotimes_{i=1}^r V_i$.  Apply to the situation where $V_i = T_{n_i}(\bfb)^{\GL(n_i)}$, and $B_i = \{ t^i_\sigma | \sigma \in S_m \}$.  Apply Proposition \ref{prop_from_SchurWeyl}.
\end{proof}
\bigskip
We now turn to the situation where $\bfb$ is not necessarily $0^m1^m$.  If $M$ is an $m$-element subset of $\{1, \cdots, 2m \}$ define
\[
    \bfb_M = (b_1, \cdots, b_{2m}) \mbox{ where: }
    b_j = \left\{
            \begin{array}{ll}
              1, & \hbox{$j \in M;$} \\
              0, & \hbox{$j \notin M.$}
            \end{array}
          \right.
\]  Let $\ga_M \in S_{2m}$ denote a (indeed any\footnote{Note that there are $(m!)^2$ possible choices for $\gamma_M$}) permutation of $\{1, \cdots, 2m \}$ such that $\ga_M$ permutes the coordinates of $0^m1^m$ to obtain $b_M$.
\begin{lemma}\label{lemma_other_b_iso}  For any $m$-element subset, $M$, of $\{1, \cdots, 2m\}$ there exists an isomorphism of $\GL(n)$-representations,
\[
    \Psi^M_n : T_{n}( 0^m1^m ) \rightarrow T_{n}( \bfb_M  )
\]
\end{lemma}
\begin{proof}  Permuting tensor factors does not change the isomorphism class of an inner tensor product.  Define $\Psi^M_n$ so as to permute the tensor factors using $\ga_M$.
\end{proof}
\begin{prop}\label{prop_other_b_basis_Tn}
Given $m$ and $n \geq m$, we have that for any $m$-element subset, $M$, of $\{1, \cdots, 2m \}$, a basis for the $\GL(n)$-invariants in $T_{n}(\bfb_M)$ is given by
    \[\left\{\Psi^M_n \left(t_n(\sigma) \right)| \sigma \in S_m \right\}. \]
Otherwise, outside of these inequalities, the above are a spanning set.
\end{prop}
\begin{proof}
Follows immediately from the isomorphism $\Psi^M_n$ (Lemma \ref{lemma_other_b_iso}) and the statement of Proposition \ref{prop_from_SchurWeyl}.
\end{proof}
For $\bfs = (\sigma_1, \cdots, \sigma_r) \in S_m^r$ define
  $t^M_{\bfn}(\bfs) = \Psi^M_{n_1}(t_{n_1}(\sigma_1)) \ot \cdots \ot \Psi^M_{n_r}(t_{n_r}(\sigma_r))$.
\begin{cor}\label{cor_other_b_basis_bfn}  Given $m$ and $\bfn$ such that for all $i$ with $1\leq i \leq r$ we have $n_i \geq m$, then for all $m$-element subsets, $M$, of $\{1, \cdots, 2m\}$ the set $\{ t^M_{\bfn}(\bfs) | \bfs \in S_m^r \}$ is a basis for the $\G$-invariants in $\bigotimes_{i=1}^r T_{n_i}(\bfb_M)$.  Otherwise, outside these inequalities, the above are a spanning set.
\end{cor}
\begin{proof}
Apply the isomorphism $\Psi^M_{n_i}$ to each tensor factor in Corollary \ref{cor_01b_basis_bfn}.
\end{proof}

We note that when $\bfb = \bfb_M$,  $t^M_\bfn(\bfs)$ is in the range of $\Phi^\bfb_\bfn$, defined above Lemma \ref{lemma_equal_01}. Finally, set $\phi^M_{\bfn}(\bfs) = (\Phi^{\bfb_M}_\bfn)^{-1}\left( t^M_{\bfn}(\bfs) \right)$.  We obtain
\begin{cor}\label{cor_basis_T(n,b)}  Given $m$ and $\bfn$ such that for all $i$ with $1\leq i \leq r$ we have $n_i \geq m$, a basis for the $\G$-invariants in $\mcT{2m}$ is
\[
    \left\{ \phi^M_{\bfn}(\bfs) | \bfs \in S_m^r \mbox{ and }M \subseteq \{1,\cdots,2m\} \mbox{ with }|M|=m \right\}
\] Otherwise, outside of these inequalities, the above are a spanning set.
  Therefore, 
\[
    \dim \left(\mcT{2m} \right)^\G \leq  \binom{2m}{m}(m!)^r.
\] with equality holding exactly when $n_i \geq m$ for all $1 \leq i \leq r$.
\end{cor}

\section{Invariants in symmetric algebra}\label{sec_symtensor}

For a complex vector space $V$, let $\bot V = \bigoplus_{d \geq 0} \bot^d V$ and $\mcS(V) = \bigoplus_{d=0}^\infty \mcS^d(V)$ denote the $\bbN$-graded tensor and symmetric algebras respectively.  Recall that $\mcS(V)$ is defined as the quotient of the tensor algebra by the two sided ideal, $\langle x\ot y-y\ot x | x,y \in V \rangle$.  The map $p:\bot(V) \rightarrow \mcS(V)$ defined by
\[
   p( x_1 \ot x_2 \ot \cdots \ot x_d )= x_1 \cdot x_2 \cdot \cdots \cdot x_d
\] on $\bot^d V$ defines a surjective homomorphism of graded associative $\bbC$-algebras.  Note, of course, that the product on $\bot(V)$, denoted by $\ot$, is non-commutative, while the product on $\mcS(V)$, denoted by $\cdot$, is commutative.

\begin{prop}\label{prop_spanning_invariants} Given $m$ and $\bfn$ such that for all $1\leq i \leq r$ we have $n_i \geq m$, the set
\[
    \left\{ p\left( \phi^M_{\bfn}(\bfs)  \right) | \bfs \in S_m^r \mbox{ and }M \subseteq \{1,\cdots,2m\} \mbox{ with }|M|=m \right\}
\] spans $\mcS^{2m}(\V \oplus \V^*)^\G$.
\end{prop}
\begin{proof}
The group $\G$ is reductive, therefore the restriction of $p$ to the $\G$-invariants maps onto the $\G$-invariants in $\mcS^{2m}(\V \oplus \V^*)$.
\end{proof}

In light of Proposition \ref{prop_spanning_invariants}, we now investigate the dependence of $p\left( \phi^M_{\bfn}(\bfs)  \right)$ on the subset $M$.  In fact we have

\begin{prop}\label{prop_no_M_needed}  For all $m$ and $\bfn$, if $M$ is an $m$-element subset of $\{1, 2, \cdots, 2m\}$ then for all $\bfs \in S_m^r$ we have
\[
    p\left( \phi^{(0^m 1^m)}_{\bfn}(\bfs)  \right) = p\left( \phi^{M}_{\bfn}(\bfs)  \right).
\]
\end{prop}
\begin{proof}  In general, if $V$ is a vector space and $v_1, \cdots, v_d \in V$ then
we have $p(v_1 \ot \cdots \ot v_d) =p(v_{\sigma(1)} \ot \cdots \ot v_{\sigma(d)})$ for any permutation $\sigma \in S_d$.  The map $(\Phi^{\bfb_M}_\bfn)^{-1} \circ \left( \Psi^M_{n_1} \ot \cdots \ot \Psi^M_{n_r} \right) \circ \Phi^{\bfb_M}_\bfn$ defines a permutation of tensor factors according to $\gamma_M$.
\end{proof}

Given $\bfs = (\sigma_1, \cdots, \sigma_r) \in S_m^r$ we define $F_{\bfs} = p( \phi^{(0^m1^m)}_\bfs )$ which is equal to the product
\begin{equation}\label{eqn_bigproduct}
\begin{array}{cc}
    & \left[\left( e_{i^{         (1 )}_1} \ot \cdots \ot e_{i^{         (1 )}_r}\right) \cdot
     \left( e^{i^{(\sigma_1(1))}_1} \ot \cdots \ot e^{i^{(\sigma_{r-1}(1))}_{r-1}} \ot e^{i^{(\sigma_r(1))}_r}\right)\right] \cdot \\
    & \vdots \\
    & \cdot \left[\left( e_{i^{         (m )}_1} \ot \cdots \ot e_{i^{         (m )}_r}\right) \cdot
     \left( e^{i^{(\sigma_1(m))}_1} \ot \cdots \ot e^{i^{(\sigma_{r-1}(m))}_{r-1}} \ot e^{i^{(\sigma_r(m))}_r}\right)\right]
\end{array}
\end{equation}
summed over all ordered $m$-tuples of $r$-tuples of indices $i^{(j)}_{k}$ with $1 \leq i^{(j)}_{k} \leq n_k$ where $1 \leq j \leq m$ and $1 \leq k \leq r$.

The product, $\cdot$, is commutative.  A consequence of this fact is that the left (resp. right) factor in the rows may be permuted.  That is we may replace $(\sigma_1, \cdots, \sigma_r)$ with $(\ga \sigma_1, \cdots, \ga \sigma_r)$ for any $\ga \in S_m$.  If $\ga = \sigma_r^{-1}$, we may reduce to the assumption that $\sigma_r = 1$ (the identity permutation).  This means that the product \ref{eqn_bigproduct} is equal to:

\begin{equation}\label{eqn_bigproduct2}
\begin{array}{cc}
    & \left[\left( e_{i^{         (1 )}_1} \ot \cdots \ot e_{i^{         (1 )}_r}\right) \cdot
     \left( e^{i^{(\sigma_1(1))}_1} \ot \cdots \ot e^{i^{(\sigma_{r-1}(1))}_{r-1}} \ot e^{i^{(1)}_r}\right)\right] \cdot \\
    & \vdots \\
    & \cdot \left[\left( e_{i^{         (m )}_1} \ot \cdots \ot e_{i^{         (m )}_r}\right) \cdot
     \left( e^{i^{(\sigma_1(m))}_1} \ot \cdots \ot e^{i^{(\sigma_{r-1}(m))}_{r-1}} \ot e^{i^{(m)}_r}\right)\right]
\end{array}
\end{equation}
Note that \emph{any one} of the permutations $\sigma_1, \cdots, \sigma_r$ could be assumed to be the identity.  We have arbitrarily chosen $\sigma_r = 1$.

The domain of the permutations $\sigma_1, \cdots, \sigma_r$ is the set $\{1, \cdots, m \}$.  Noting again that the product, $\cdot$, in the symmetric algebra is commutative, we see that the ``rows'' of the above expression may be permuted without changing the expression.  Permuting the rows corresponds to a simultaneous permutation of the domain of each $\sigma_k$.  That is, we may simultaneously conjugate $\sigma_1, \cdots, \sigma_r$ (since relabeling a permutation's domain corresponds to conjugating the permutation).  We obtain that for any $\bfs=(\sigma_1, \cdots, \sigma_r) \in S_m^r$ and $\tau \in S_m$ we have $F_{(\sigma_1, \cdots, \sigma_r)} = F_{(\tau \sigma_1 \tau^{-1}, \cdots, \tau \sigma_r \tau^{-1})}$.

The commutative $\bbC$-algebra of polynomial functions on a vector space $\bbV$ is naturally isomorphic to $\mcS(\bbV^*)$.  If $\bbV = \V \oplus \V^*$ then $\bbV$ is self-dual.  Therefore, we obtain a natural isomorphism between $\mcS(\V \oplus \V^*)$ and $\mcP( \V \oplus \V^*)$.

Next, we explicitly describe the value of each $F_\bfs$ on the vector space $\V \oplus \V^*$.
A general elements $\V$ is of the form
\[
    X = \sum_{ \twoline{1 \leq i_j \leq n_j}{1 \leq j \leq r}}
    x^{i_1 i_2 \cdots i_r} e_{i_1} \ot \cdots \ot e_{i_r},
\]
where $x^{i_1 i_2 \cdots i_r} \in \bbC$ while a general element of $\V^*$ is of the form
\[
    Y = \sum_{ \twoline{1 \leq i_j \leq n_j}{1 \leq j \leq r}}
    y_{i_1 i_2 \cdots i_r} e^{i_1} \ot \cdots \ot e^{i_r}.
\] where $y_{i_1 i_2 \cdots i_r}\in \bbC$.
The value of $\ot_{j=1}^r e^{k_j} \in \mcP^1(\V)$ on $\V \oplus \V^*$ is $x^{j_1 j_2 \cdots j_r}$, while the value of $\ot_{j=1}^r e_{k_j} \in \mcP^1(\V^*)$ on $\V \oplus \V^*$ is $y_{i_1 i_2 \cdots i_r}$.

Therefore, given $(\sigma_1, \cdots, \sigma_{r-1}) \in S_m^{r-1}$, let $f_\bfs$ denote the polynomial function on $\V \oplus \V^*$ corresponding to $F_\bfs$ with $\bfs = (\sigma_1, \cdots, \sigma_{r-1}, \sigma_r)$ with $\sigma_r = 1$.  That is, the value of $f_\bfs$ on $(X,Y)$ is
\[
     \left[\left( x^{i^{         (1 )}_1 \cdots i^{         (1 )}_r}\right)
     \left( y_{i^{(\sigma_1(1))}_1 \cdots i^{(\sigma_r(1))}_r}\right)\right]
     \cdots
     \left[\left( x^{i^{         (m )}_1 \cdots i^{         (m )}_r}\right)
     \left( y_{i^{(\sigma_1(m))}_1 \cdots i^{(\sigma_r(m))}_r}\right)\right]
\] summed over all ordered $m$-tuples of $r$-tuples of indices $i^{(j)}_{k}$ with $1 \leq i^{(j)}_{k} \leq n_k$ where $1 \leq j \leq m$ and $1 \leq k \leq r$.

\begin{thm}\label{thm_basis_poly} Given $m$ and $r$, let $S_m^{r-1} / S_m$ denote the orbits of $S_m$ on the set $S_m^{r-1}$ under the action of simultaneous conjugation.  Then let $t = |S_m^{r-1} / S_m|$, and choose distinct representatives from each $S_m$-orbit, $\bfs_1, \cdots \bfs_t$.  For any $\bfn$ with $n_i \geq m$ for all $1\leq i \leq r$, the set
\[ \mc B = \{ f_{\bfs_1}, f_{\bfs_2}, \cdots, f_{\bfs_t} \} \]
is a basis for the $\G$-invariants in $\mcS^{2m}(\V \oplus \V^*)$.  Otherwise, outside of these inequalities, the above are a spanning set.
\end{thm}
\begin{proof}
We have seen that $\mc B$ spans.  In \cite{HeroWillenbring} it is shown that $t = \widetilde h_{m,r} = \dim  \mcS^{2m}(\V \oplus \V^*)^\G$.  Therefore, $\mc B$ is linearly independent.
\end{proof}

\def\cprime{$'$} \def\cprime{$'$}

\end{document}